\documentclass[12pt]{article}

\usepackage{graphicx}
\usepackage{comment}
\usepackage{mathrsfs}

\usepackage{nameref}
\usepackage{blkarray}
\usepackage[shortlabels,inline]{enumitem}
\usepackage{amsmath}
\usepackage{amssymb}
\usepackage{empheq}
\usepackage{tikz}

\usepackage[shortlabels]{enumitem}
\setlist[enumerate]{nosep}
\usepackage[doc]{optional}
\usepackage{xcolor}
\usepackage[colorlinks=true,
            linkcolor=refkey,
            urlcolor=lblue,
            citecolor=red]{hyperref}
\usepackage{float}
\usepackage{soul}
\usepackage{graphicx}
\definecolor{labelkey}{rgb}{0,0.08,0.45}
\definecolor{refkey}{rgb}{0,0.6,0.0}
\definecolor{Brown}{rgb}{0.45,0.0,0.05}
\definecolor{lime}{rgb}{0.00,0.8,0.0}
\definecolor{lblue}{rgb}{0.5,0.5,0.99}

\usepackage{mathpazo}
\usepackage{blkarray}
\usepackage{subcaption}
\colorlet{hlcyan}{cyan!30}

\usepackage{stmaryrd}
\usepackage{amssymb}
\makeatletter
\def\namedlabel#1#2{\begingroup
   \def\@currentlabel{#2}%
   \label{#1}\endgroup
}
\makeatother

\newcommand{\R}{\mathbb{R}}

\renewcommand{\implies}{\Rightarrow}
\renewcommand{\iff}{\Leftrightarrow}

  \newcommand{\eps}{\varepsilon}
  \renewcommand{\epsilon}{\varepsilon}
  \renewcommand{\phi}{\varphi}

  \renewcommand{\ker}{\operatorname{\mathrm{ker}}}

  \newcommand{\Id}{\mathrm{Id}}

  \DeclareMathOperator{\co}{\mathrm{co}}

  \newcommand{\norm}[1]{\|#1\|}

  \def\llangleT{\langle\kern-3.5pt\langle}
  \def\rrangleT{\rangle\kern-3.5pt\rangle}

  \makeatletter
  \def \weaktostar@sym{\setbox0=\hbox{$\rightharpoonup$}\rlap{\hbox 
  to\wd0{\hss\raise1ex\hbox{$\scriptscriptstyle{*\,}$}\hss}}\box0}
  \def \weaktostar    {\mathrel{\weaktostar@sym}}
  \makeatother

  \def\gap{\mathcal{G}}

  \DeclareMathOperator{\dist}{dist}

  \def\DerivConeSym{V}
  \def\DerivConeDSym{W}
\newcommand{\DerivCone}[1][\relax]{\DerivConeSym\ifx\relax#1\relax\else(#1)\fi}
\newcommand{\DerivConePrime}[1][\relax]{\tilde\DerivConeSym\ifx\relax#1\relax\else(#1)\fi}
\newcommand{\DerivConeX}[2][]{\DerivConeSym_{#2}\ifx\relax#1\relax\else(#1)\fi}

\newcommand{\DerivConeDXt}[3][]{\DerivConeDSym_{#3}^{#2}\ifx\relax#1\relax\else(#1)\fi}
\newcommand{\DerivConeDXtp}[3][]{\DerivConeDSym_{#3}^{\circ,#2}\ifx\relax#1\relax\else(#1)\fi}

\newcommand{\localRhoX}[1][]{\if\relax\detokenize{#1}\relax r_x\else r_{x,#1}\fi}
\newcommand{\localRhoY}[1][]{\if\relax\detokenize{#1}\relax r_y\else r_{y,#1}\fi}

\def\weaklim{\mathop{\operatorname{w-\kern.07em lim}\,}}
\def\weaklimsup{\mathop{\operatorname{w-\kern.07em lim\,sup}\,}}
\def\weakliminf{\mathop{\operatorname{w-\kern.07em lim\,inf}\,}}
\def\weakstarlim{\mathop{\operatorname{w-\!\ast\!-\kern.07em lim}\,}}
\def\weakstarlimsup{\mathop{\operatorname{w-\!\ast\!-\kern.07em lim\,sup}\,}}
\def\weakstarliminf{\mathop{\operatorname{w-\!\ast\!-\kern.07em lim\,inf}\,}}

\newcommand{\menge}[2]{\big\{{#1}~\big |~{#2}\big\}}

\newcommand{\fenv}[1]%
{\ensuremath{\,\overrightarrow{\operatorname{env}}_{#1}}}
\newcommand{\benv}[1]%
{\ensuremath{\,\overleftarrow{\operatorname{env}}_{#1}}}

\DeclarePairedDelimiterX\set[2]{ \{ }{ \}_{#2} }{#1}

\DeclarePairedDelimiterX\rb[1]{ ( }{ ) }{#1}

{\begin{list}{}{%
\settowidth{\labelwidth}{\textrm{#1~}}%
\setlength{\leftmargin}{\labelwidth+\labelsep}}}
{\end{list}}
\usepackage{amsthm}
\usepackage[capitalize,nameinlink]{cleveref}
\crefname{equation}{}{equations}
\crefname{chapter}{Appendix}{chapters}
\crefname{item}{}{items}
\crefname{enumi}{}{}

\theoremstyle{definition}
\newtheorem{theorem}{Theorem}[section]
\newtheorem{lemma}[theorem]{Lemma}

\newtheorem{corollary}[theorem]{Corollary}

\newtheorem{definition}[theorem]{Definition}

\newtheorem{ex}[theorem]{Example}

\newtheorem{remark}[theorem]{Remark}

\newtheorem{conjecture}[theorem]{Conjecture}

\usepackage{multirow}

\usepackage[customcolors]{hf-tikz}
\usetikzlibrary{calc, intersections}
\usepackage{sidecap}
 \usepgflibrary{shapes.multipart}
 \usetikzlibrary{decorations.pathreplacing}
\usetikzlibrary{fadings}
\usepgflibrary{decorations.text}

\tikzset{style green/.style={
    set fill color=green!50!lime!60,
    set border color=white,
  },
  style cyan/.style={
    set fill color=cyan!90!blue!60,
    set border color=white,
  },
  style gray/.style={
    set fill
    color=black!10!,
    set border color=white,
  },
  style orange/.style={
    set fill color=orange!80!red!60,
    set border color=white,
  },
  hor/.style={
    above left offset={-0.15,0.31},
    below right offset={0.15,-0.125},
    #1
  },
  ver/.style={
    above left offset={-0.1,0.3},
    below right offset={0.15,-0.15},
    #1
  }
}

\newcommand{\boxedeqn}[1]{%
    \begin{align}\fbox{%
        \addtolength{\linewidth}{-2\fboxsep}%
        \addtolength{\linewidth}{-2\fboxrule}%
        \begin{minipage}{\linewidth}%
        \begin{equation}#1\\begin{align}+5mm]\end{equation}%
        \end{minipage}%
      }\end{align}%
  }

\providecommand{\RA}{\Rightarrow}

\providecommand{\Id}{\operatorname{{ Id}}}

\usepackage{algorithm,algorithmic} 
\usepackage{hyperref}
\definecolor{myblue}{rgb}{.8, .8, 1}

\newcommand*\circled[2][1.6]{\tikz[baseline=(char.base)]{
    \node[shape=circle, draw, inner sep=1pt, 
        minimum height={\f@size*#1},] (char) {\vphantom{WAH1g}#2};}}
\makeatother

\allowdisplaybreaks 

\newtheorem{assumption}[theorem]{Assumption}

\newcommand{\PA}{P_A}
\newcommand{\PB}{P_B}
\renewcommand{\RA}{R_A}

\renewcommand{\norm}[1]{\left\lVert #1 \right\rVert}
\renewcommand{\gap}{g}

\usepackage[T1]{fontenc}

\hypersetup{
  pdftitle={The Flow-Limit of Reflect-Reflect-Relax: Existence, Stability, and Discrete-Time Behavior},
  pdfauthor={Manish Krishan Lal}
}

\begin{document}

\title{\textsc{The Flow-Limit of Reflect-Reflect-Relax: Existence, Stability, and Discrete-Time Behavior}}

\author{
Manish\ Krishan Lal\thanks{
                 Mathematics, 	
Technical University of Munich,\ Garching, 85748, Germany.
                 E-mail: \texttt{manish.krishanlal@tum.de}.}
                 }

\date{\today} 
\maketitle

\begin{abstract} \noindent
We study the Reflect–Reflect–Relax (RRR) algorithm in its small–step (flow–limit)
regime. In the smooth transversal setting, we show that the transverse dynamics
form a \emph{hyperbolic sink}, yielding \emph{exponential decay} of a natural gap
measure. Under uniform geometric assumptions, we construct a
\emph{tubular neighborhood} of the feasible manifold on which the squared gap
defines a \emph{strict Lyapunov function}, excluding recurrent dynamics and
\emph{chaotic behavior} within this basin.

In the discrete setting, the induced flow is piecewise constant on
\emph{W-domains} and supports \emph{Filippov sliding} along convergent boundaries,
leading to finite-time capture into a solution domain. We prove that small-step
RRR is a \emph{forward–Euler discretization} of this flow, so that solution times
measured in rescaled units converge to a finite limit while iteration counts
diverge, explaining the \emph{emergence of iteration-optimal relaxation
parameters}. Finally, we introduce a \emph{heuristic mesoscopic framework} based
on \emph{percolation} and \emph{renormalization group} to organize performance
deterioration near the Douglas–Rachford limit.

\end{abstract}

{
\noindent
{\bfseries 2020 Mathematics Subject Classification:}
Primary 37N40; Secondary 49J53, 90C30, 65K10.

\noindent {\bfseries Keywords:}
Reflect-Reflect-Relax, flow limit, nonconvex feasibility, dynamical systems,
Filippov sliding, Lyapunov stability
}

\section{Introduction}

Projection and reflection methods such as Douglas Rachford splitting and the
Reflect-Reflect-Relax (RRR) algorithm have proved effective for feasibility
problems involving nonconvex and even discrete constraints. These methods are
used in a wide range of applications, including combinatorial design problems,
constraint satisfaction, and learning formulations cast as feasibility problems \cite{elser2021learning}.
Despite their empirical success, the mechanisms governing their convergence and
their dependence on algorithmic parameters remain only partially understood.

Motivated by extensive numerical experiments, Elser proposed a flow limit
viewpoint for RRR, in which progress is measured not in raw iteration counts but
in rescaled time units of the form $t = k\eps$, where $\eps>0$ denotes the step
size and $x_k$ the $k$th iterate. From this viewpoint, a robust empirical
phenomenon emerges: as $\eps$ decreases, the solution time measured in rescaled
units often stabilizes to a finite value, while the number of iterations required
to reach a solution increases and exhibits a problem dependent optimal choice of
$\eps$. This behavior is observed across a range of nonconvex feasibility
problems, including those studied in \cite{elser2025solving}.

The purpose of this paper is to provide a rigorous mathematical foundation for
this small step flow picture. We focus on two regimes that arise naturally in
practice. The first is a smooth regime, in which the constraint sets are smooth
manifolds intersecting transversally. The second is a discrete regime, in which
the constraint sets are finite and the induced dynamics exhibit sharp switching
behavior. In both cases, we interpret RRR as a discretization of an underlying
continuous time system and analyze the structure of the resulting flow. To this end, we define the RRR flow field
\begin{equation*}
v(x) := \PB(\RA x) - \PA x, \qquad \RA := 2\PA - \Id,
\end{equation*}
and study the continuous dynamics $\dot x = v(x)$ together with its small step
discretization
\[
x_{k+1} = x_k + \eps\,v(x_k).
\]
This formulation allows us to treat smooth and nonsmooth behavior within a
single framework and to make precise comparisons between continuous trajectories
and discrete iterates.

\paragraph{Contributions and scope.}
All of our rigorous results are established under local assumptions. In the
smooth setting, they are proved within a positively invariant tubular
neighborhood of the feasible manifold on which the metric projections are
single valued. In the discrete setting, they are proved along specified chains
of convergent W-domain boundaries. Within this framework, we establish the
following results.

\begin{enumerate}[(i)]
  \item \textbf{Transverse stability from the flow linearization.}
In the smooth transversal regime, we use the explicit linearization of the RRR
flow in terms of tangent space projectors and prove that the transverse dynamics
form a hyperbolic sink, with exponential decay rates determined by principal
angles (Lemma~\ref{lem:jacobian}, Theorem~\ref{thm:smooth}).
 \item \textbf{Piecewise constant flow and sliding capture.}
  For finite constraint sets, the induced flow is piecewise constant on
  W-domains. Along two cell boundaries satisfying a convergent normal condition,
  the associated Filippov sliding dynamics drive trajectories into a solution
  cell in finite time
  (Theorem~\ref{thm:discrete}).

  \item \textbf{Euler approximation of the flow.}
  Small step RRR is an $O(\eps)$ accurate forward Euler discretization of the
  continuous flow on finite time intervals, including across sliding segments,
  yielding $O(\eps)$ errors in trajectories
  (Theorem~\ref{thm:euler}).

  \item \textbf{Flow limit behavior of solution times.}
  In the feasible case, the continuous solution time to reach a fixed gap
  threshold is finite, and the discrete solution time converges to this limit as
  $\eps\downarrow 0$, while the corresponding iteration count diverges like
  $1/\eps$
  (Theorems~\ref{thm:hitting} and~\ref{thm:quant}).

  \item \textbf{Lyapunov structure and absence of recurrence.}
  Under uniform geometric assumptions, the squared gap defines a strict
  Lyapunov function on a tubular neighborhood of the feasible manifold,
  excluding nontrivial recurrence and chaotic dynamics within this basin
  (Theorem~\ref{thm:tube-Lyap}, Corollary~\ref{cor:no-recurrence}).
\end{enumerate}

Beyond these local results, we introduce a formal but heuristic mesoscopic
framework based on percolation and renormalization group to organize the
empirical deterioration of RRR performance as the relaxation parameter
approaches the Douglas Rachford limit. We do not claim global convergence or a complete description of RRR dynamics for
arbitrary step sizes. Instead, our results provide a rigorous foundation for the
small step regime of the flow limit conjecture and for the absence of chaotic
behavior near feasible solutions.

\section{Preliminaries}
\subsection{RRR and the flow field}
Let $A,B\subset\R^m$ be closed sets with (possibly multivalued) metric projections 
$P_A,P_B$. The RRR update with step $\eps>0$ is
\begin{equation}\label{eq:rrr}
x_{k+1} \;=\; x_k \;+\; \eps\,\bigl(\PB(\RA x_k)-\PA x_k\bigr).
\end{equation}
where $\RA \;=\; 2\PA - \Id$ is the reflector. Formally, as $\eps\to 0$ and $t=k\eps$, the piecewise-linear interpolation of 
$\{x_k\}$ approaches the solution of the ODE
\begin{equation}\label{eq:flow}
\dot x \;=\; v(x)\;=\; \PB(\RA x)-\PA x.
\end{equation}
We call $v$ the \emph{RRR flow field}. A point $x$ is a (RRR) fixed point of the flow 
if $v(x)=0$. We write the \emph{gap}
\begin{equation}\label{eq:gap}
    \gap(x) \;:=\; \norm{v(x)},
\end{equation}
as our basic measure of progress; in the smooth setting, $\gap(x)$ is 
equivalent to the transverse distance from $x$ to $A\cap B$ in a neighborhood of 
a transversal intersection.

\subsection{Regularity classes}
We will use two basic regimes:

\begin{itemize}
  \item \textbf{(S1) Smooth transverse case.}  
  $A$ and $B$ are $C^2$ embedded submanifolds of $\R^m$ that intersect 
  transversally at $x^*\in A\cap B$. In a sufficiently small neighborhood $U$ of 
  $x^*$ the metric projections $P_A$ and $P_B$ are single-valued and $C^1$, and 
  their derivatives at $x^*$ are the orthogonal projectors $p_A,p_B$ onto the 
  tangent spaces $\mathcal{T}_{x^*}A,\mathcal{T}_{x^*}B$ respectively. The intersection 
  $M:=A\cap B$ is then a $C^2$ submanifold, and $P_A,P_B$ restrict to the identity 
  on $M$ near $x^*$.
  \item \textbf{(S2) Discrete W-domain case.}  
  $A$ and $B$ are finite subsets of $\R^m$. Then $P_A$ and $P_B$ are single-valued 
  everywhere (nearest neighbors are unique except on a set of measure zero), and 
  the induced velocity field $v$ is piecewise constant on a finite partition of 
  \emph{W-domains}. For more general stratified sets, one typically 
  obtains a piecewise $C^1$ velocity field; here we focus on the finite case where 
  the geometry is purely combinatorial.
\end{itemize}

\subsection{Principal angles and the Jacobian at a regular point}
In the smooth transverse case (S1), the linearization of the RRR flow at a feasible point admits an explicit expression in terms of the orthogonal projectors onto the
tangent spaces of the constraint manifolds. Let $p_A$ and $p_B$ denote the orthogonal 
projectors onto the tangent spaces $\mathcal{T}_{x^*}A$ and $\mathcal{T}_{x^*}B$ at a transversal 
intersection $x^*\in A\cap B$. The direct sum $\mathcal{T}_{x^*}A+\mathcal{T}_{x^*}B$ admits an 
orthogonal decomposition into principal-angle planes and the common tangent 
subspace $\mathcal{T}_{x^*}A\cap \mathcal{T}_{x^*}B$. On each principal-angle plane the angle between 
$\mathcal{T}_{x^*}A$ and $\mathcal{T}_{x^*}B$ is some $\theta\in(0,\tfrac{\pi}{2}]$.

\begin{lemma}[Linearization of the flow at a feasible point]\label{lem:jacobian}
Assume \textnormal{(S1)} and let $x^*\in A\cap B$ be a transversal intersection. 
Then $x^*$ is an equilibrium of the flow~\eqref{eq:flow}, and its Jacobian is
\begin{equation}\label{eq:jac}
Dv(x^*) \;=\; D\PB(x^*)\,D\RA(x^*) - D\PA(x^*) 
\;=\; \bigl(2p_Bp_A - p_B - p_A\bigr).
\end{equation}
Moreover:
\begin{enumerate}[(i)]
  \item\label{lem:jacobian1} On the intersection $\mathcal{T}_{x^*}A\cap \mathcal{T}_{x^*}B$ we have $Dv(x^*)=0$.
  \item\label{lem:jacobian2} On each principal-angle $2$-plane with angle $\theta\in(0,\frac{\pi}{2}]$, 
  there exists an orthonormal basis in which
  \[
  p_A=\begin{bmatrix}1&0\\0&0\end{bmatrix},\qquad
  p_B=\begin{bmatrix}\cos^2\theta&\cos\theta\sin\theta\\[2pt]
                     \cos\theta\sin\theta&\sin^2\theta\end{bmatrix},
  \]
  and in this basis the restriction of $Dv(x^*)$ is
  \[
  r(\theta) \;=\; 
  \begin{bmatrix}
    -\sin^2\theta & -\sin\theta\cos\theta\\[2pt]
    \sin\theta\cos\theta & -\sin^2\theta
  \end{bmatrix}.
  \]
  The eigenvalues of $r(\theta)$ are $-\sin^2\theta\pm i\,\sin\theta\cos\theta$, 
  and its symmetric part equals $-\sin^2\theta\,I_2$.
\end{enumerate}
\end{lemma}

\begin{proof}
Since $P_A(x^*)=P_B(x^*)=x^*$, we have $\RA(x^*)=x^*$ and hence $v(x^*)=0$. On $U$ 
the chain rule gives $D\RA(x^*)=2DP_A(x^*)-I=2p_A-I$, $D\PB(x^*)=p_B$, and 
$D\PA(x^*)=p_A$, which yields~\eqref{eq:jac}. \ref{lem:jacobian1} follows from 
$p_Az=p_Bz=z$ for $z\in \mathcal{T}_{x^*}A\cap \mathcal{T}_{x^*}B$. The block decomposition in \ref{lem:jacobian2} is 
standard principal-angle geometry: in the stated basis one computes $r(\theta)$ by 
direct multiplication of $2p_Bp_A-p_B-p_A$, and the eigenvalues and symmetric part 
follow from a short calculation.
\end{proof}

Thus $M=A\cap B$ is a manifold of equilibria, and on the transverse complement of 
$\mathcal{T}_{x^*}M$ the linearization has eigenvalues with strictly negative real part, 
forming a stable spiral when $0<\theta<\frac{\pi}{2}$ and a stable node when 
$\theta=\frac{\pi}{2}$.

\subsection{W-domains and piecewise-constant fields}
\label{sec:Wdomains}
In the discrete case (S2) we assume $A$ and $B$ are finite sets. For each pair 
$(a,b)\in A\times B$ define the corresponding \emph{W-domain}
\begin{align}\label{def:Wdomain}
W(a,b) := 
\menge{x\in\R^m}{\PA(x)=a,~ \PB(\RA x)=b}.
\end{align}
Away from a set of measure zero where projections are nonunique, the family 
$\{W(a,b)\}_{(a,b)\in A\times B}$ forms a partition of $\R^m$ into finitely many 
cells. On such a cell we have
\[
v(x) \;=\; \PB(\RA x)-\PA x \;\equiv\; b-a,\qquad x\in W(a,b),
\]
so the velocity field is piecewise constant. The interfaces between W-domains are smooth codimension-one manifolds almost everywhere, given by Voronoi bisectors and their reflections, with lower-dimensional junction sets where three or more domains intersect.

\subsection{Filippov solutions and sliding}
The piecewise-constant field $v$ is discontinuous on W-domain boundaries. We use 
the notion of \emph{Filippov solutions}~\cite{Filippov1988}. For a measurable, locally bounded vector field 
$v\colon\R^m\to\R^m$, its \emph{Filippov regularization} at $x$ is the closed convex 
hull $\mathcal F(x)$ of all essential values of $v$ near $x$. A \emph{Filippov 
solution} of
\[
\dot x(t) \in \mathcal F(x(t))
\]
is an absolutely continuous curve $x(\cdot)$ satisfying the inclusion almost 
everywhere. Consider a smooth portion of a boundary $\Sigma$ separating two W-domains with 
constant velocities $v_1$ and $v_2$. Let $\Sigma=\partial W(a_1,b_1)\cap\partial W(a_2,b_2)$ be a $C^1$ two-cell
interface, and let $n$ denote the unit normal to $\Sigma$ oriented from
$W(a_1,b_1)$ into $W(a_2,b_2)$. The Filippov set on $\Sigma$ is
\[
\mathcal F(x)=\co\{v_1,v_2\},\qquad v_i:=b_i-a_i .
\]
If the \emph{convergent normal condition}
\begin{equation}\label{eq:conv-normal}
(n\!\cdot\! v_1)\,(n\!\cdot\! v_2)<0
\end{equation}
holds, there exists a unique $v_{\mathrm{slide}}\in\mathcal F(x)$ satisfying
$n\cdot v_{\mathrm{slide}}=0$. Writing
\[
v_{\mathrm{slide}}
=\alpha v_1+(1-\alpha)v_2,\qquad
\alpha=\frac{n\cdot v_2}{n\cdot (v_2-v_1)}\in(0,1),
\]
this is equivalently expressed as
\begin{equation}\label{def:vslide}
v_{\mathrm{slide}}
:=\frac{(n\cdot v_2)v_1-(n\cdot v_1)v_2}{n\cdot (v_2-v_1)}.
\end{equation}
Under \eqref{eq:conv-normal}, $\Sigma$ is locally attracting for Filippov
solutions, which evolve along $\Sigma$ with velocity $v_{\mathrm{slide}}$
\cite{Brogliato2016}. For a $C^1$ interface $\Sigma$ and $x\in\Sigma$, let
$\Pi_{\Sigma}(x)$ denote the orthogonal projector onto the tangent space
$T_x\Sigma$. We say that the sliding velocity satisfies a
\emph{uniform tangential speed bound} on $\Sigma$ if
\begin{equation}\label{eq:tang-speed}
\|\Pi_{\Sigma}(x)\,v_{\mathrm{slide}}\|\ \ge\ s_0 \ >\ 0,
\qquad \forall x\in\Sigma .
\end{equation}

\subsection{Solution time in continuous and discrete time}
Fix a tolerance $\delta>0$. For a (Carath\'eodory or Filippov) flow solution 
$x(\cdot)$ of~\eqref{eq:flow} we define the continuous \emph{hitting time}
\begin{equation}\label{eq:cont-hitting}
T^*_\delta(x_0) \;:=\; 
\inf\menge{t\ge 0}{ \gap\bigl(x(t)\bigr)\le \delta },
\end{equation}
with the convention that $T^*_\delta(x_0)=+\infty$ if the set is empty, and with
initial condition $x(0)=x_0$. We refer to $\gap(x)$ as the \emph{gap function} \cref{eq:gap}, which vanishes precisely at
RRR fixed points. For the discrete RRR iteration~\eqref{eq:rrr} with step $\eps>0$, we define the
corresponding discrete hitting time
\begin{equation}\label{eq:disc-hitting}
t^*_\delta(\eps;x_0) \;:=\; k\eps,
\end{equation}
where $k$ is the smallest index such that $\gap(x_k)\le \delta$, if such an index
exists, and $t^*_\delta(\eps;x_0)=+\infty$ otherwise. One of our main goals is to understand the dependence of
$t^*_\delta(\eps;x_0)$ on $\eps$ and its relation to the flow-limit
$T^*_\delta(x_0)$ as $\eps\downarrow 0$.

\section{Main Results}
We now state and prove our main results in the smooth transverse regime (S1) and 
the discrete W-domain regime (S2). 

\subsection{Smooth transverse case}
The local affine form of the flow suggests contracting behavior transverse to
the feasible manifold, as illustrated through low-dimensional examples
\cite[Figures~5.3-5.4]{elser2025solving}.
The following theorem makes this precise by establishing exponential decay of
the flow in the transverse directions near a transversal feasible point.

\begin{theorem}[Hyperbolic sink and exponential decay of the gap]\label{thm:smooth}
Assume \textnormal{(S1)} and let $x^*\in A\cap B$ be a transversal intersection. 
Then $x^*$ is an equilibrium of \eqref{eq:flow}, and there exist constants 
$\mu>0$, $C\ge 1$ and a neighborhood $U$ of $x^*$ such that for any solution 
$x(\cdot)$ of \eqref{eq:flow} with $x(0)\in U$,
\begin{equation}\label{eq:smooth-exp}
  \norm{v(x(t))} \;\le\; C\,e^{-\mu t}\,\norm{v(x(0))}\qquad
  \forall t\ge 0.
\end{equation}
In particular, for any $\delta>0$ and $x(0)\in U$ with $\gap(x(0))>\delta$,
\begin{equation}\label{eq:Tdelta-log}
  T^*_\delta(x(0)) \;\le\; \mu^{-1}
    \log \Bigl(\frac{C\,\norm{v(x(0))}}{\delta}\Bigr) \;<\; \infty.
\end{equation}
\end{theorem}

\begin{proof}
By \textnormal{(S1)}, the projections $P_A,P_B$ are single-valued and $C^1$ on a
neighborhood of $x^*$, hence $v=P_B\circ R_A-P_A$ is $C^1$ there. By
Lemma~\ref{lem:jacobian}, $v(x^*)=0$, so $x^*$ is an equilibrium of \eqref{eq:flow}.
Let $J:=Dv(x^*)$.

\emph{Step 1 (coercivity of the symmetric part on the transverse sum).}
Set $\mathcal T:=T_{x^*}A+T_{x^*}B$ and let $P_{\mathcal T}$ be the orthogonal
projector onto $\mathcal T$. By Lemma~\ref{lem:jacobian}\ref{lem:jacobian1}-\ref{lem:jacobian2},
on $\mathcal T$ the symmetric part
$S_*:=\tfrac12(J+J^\top)$ satisfies
\begin{equation}\label{eq:Sstar-coercive}
S_* \;\preceq\; -\sigma\,P_{\mathcal T},
\qquad
\sigma:=\min\{\sin^2\theta_j:\ \theta_j\in(0,\tfrac{\pi}{2}]\}>0,
\end{equation}
and $S_*=0$ on $\mathcal T^\perp$.

\emph{Step 2 (transverse dominance of $v$ near $x^*$).}
Since $v$ is $C^1$ and $v(x^*)=0$, we have
\begin{equation}\label{eq:v-taylor}
v(x)=J(x-x^*)+r(x),\qquad \frac{\|r(x)\|}{\|x-x^*\|}\to 0\ \ (x\to x^*).
\end{equation}
Moreover, $\mathrm{ran}(J)\subseteq \mathcal T$ by Lemma~\ref{lem:jacobian}
(since $J$ acts as a direct sum of blocks supported on $\mathcal T$), hence
$P_{\mathcal T}J=J$. Fix $\eta\in(0,1)$ and shrink $r>0$ so that
$\|r(x)\|\le \eta\,\|J(x-x^*)\|$ for all $\|x-x^*\|\le r$. Then for all such $x$,
\[
\|P_{\mathcal T}v(x)\|
=\|J(x-x^*)+P_{\mathcal T}r(x)\|
\ge (1-\eta)\|J(x-x^*)\|,
\qquad
\|v(x)\|\le (1+\eta)\|J(x-x^*)\|,
\]
and therefore
\begin{equation}\label{eq:transverse-dominance}
\|P_{\mathcal T}v(x)\|\ \ge\ \kappa\,\|v(x)\|,
\qquad
\kappa:=\frac{1-\eta}{1+\eta}\in(0,1),
\qquad \forall x\in B(x^*,r).
\end{equation}

\emph{Step 3 (Lyapunov inequality for $E=\tfrac12\|v\|^2$).}
Define $E(x):=\tfrac12\|v(x)\|^2$ on $B(x^*,r)$. Along any solution $x(\cdot)$ of
\eqref{eq:flow} with $x(t)\in B(x^*,r)$ we compute
\begin{equation}\label{eq:Edot}
\frac{d}{dt}E(x(t))
=\langle v(x(t)),Dv(x(t))v(x(t))\rangle
=\langle v(x(t)),S(x(t))v(x(t))\rangle,
\end{equation}
where $S(x):=\tfrac12(Dv(x)+Dv(x)^\top)$. By continuity of $Dv$ at $x^*$, after
shrinking $r$ we may assume
\[
\|S(x)-S_*\|\le \frac{\sigma\kappa^2}{4}\qquad \forall x\in B(x^*,r).
\]
Then for $x\in B(x^*,r)$,
\[
\langle v,S(x)v\rangle
\le \langle v,S_*v\rangle+\|S(x)-S_*\|\,\|v\|^2
\le -\sigma\|P_{\mathcal T}v\|^2+\frac{\sigma\kappa^2}{4}\|v\|^2
\le -\frac{\sigma\kappa^2}{2}\|v\|^2,
\]
where the second inequality uses \eqref{eq:Sstar-coercive} and the last uses
\eqref{eq:transverse-dominance}. Hence, along solutions staying in $B(x^*,r)$,
\[
\dot E(x(t)) \le -2\mu\,E(x(t)),
\qquad \mu:=\frac{\sigma\kappa^2}{2}>0.
\]
Gronwall yields $E(x(t))\le E(x(0))e^{-2\mu t}$, i.e.
\[
\|v(x(t))\|\le e^{-\mu t}\|v(x(0))\|,
\]
which proves \eqref{eq:smooth-exp} with $C=1$ (and thus also with any $C\ge 1$ after
shrinking $U$).

Finally, if $\|v(x(0))\|>\delta$, then the first hitting time
$T^*_\delta(x(0)):=\inf\{t\ge 0:\ \|v(x(t))\|\le \delta\}$ satisfies
$e^{-\mu T^*_\delta}\|v(x(0))\|\le \delta/C$, which is equivalent to
\eqref{eq:Tdelta-log}.
\end{proof}

\subsection{Discrete W-domains and sliding capture}
In the discrete regime (S2), where the constraint sets are finite, the RRR flow is
piecewise constant on a finite partition of the state space into W-domains,
with trajectories exhibiting attraction to domain boundaries
\cite[Section~5.5.4, Figure~5.6]{elser2025solving}.
The following result provides a precise formulation of this behavior using
Filippov solutions and establishes finite-time capture along convergent
boundaries.

\begin{theorem}[Piecewise-constant flow and sliding capture]\label{thm:discrete}
Assume \textnormal{(S2)} and $A\cap B\neq\varnothing$. Then:

\begin{enumerate}[(i)]
\item\label{thm:discrete1}
There exists a finite family $\{W(a,b)\}_{(a,b)\in A\times B}$ such that
\[
\R^m=\bigcup_{(a,b)} W(a,b)\ \ \text{up to a null set},
\qquad
v(x)=b-a \ \ \forall x\in W(a,b).
\]
For any $x_0\notin\bigcup_{(a,b)}\partial W(a,b)$, the Carath\'eodory solution of
\eqref{eq:flow} with $x(0)=x_0$ is unique until the first boundary hit.

\item\label{thm:discrete2}
Let $\Sigma=\partial W(a_1,b_1)\cap\partial W(a_2,b_2)$ be a $C^1$ interface with
unit normal $n$ and velocities $v_i=b_i-a_i$. If \eqref{eq:conv-normal} holds, then
there exists a unique $v_{\mathrm{slide}}\in\co\{v_1,v_2\}$ satisfying
$n\cdot v_{\mathrm{slide}}=0$, and $\Sigma$ is locally attracting for Filippov
solutions.

\item\label{thm:discrete3}
Let $\Sigma_1,\dots,\Sigma_N$ be $C^1$ interfaces with
\[
\Sigma_j\subset\partial W(a_j,b_j)\cap\partial W(a_{j+1},b_{j+1}),
\qquad
\Sigma_N\cap W(a^*,a^*)\neq\varnothing,
\]
and assume \eqref{eq:conv-normal} and \eqref{eq:tang-speed} hold on each
$\Sigma_j$. If $\mathrm{length}(\Sigma_j)=\ell_j<\infty$, then any Filippov
solution entering $\Sigma_1$ reaches $W(a^*,a^*)$ in time
\[
T_{\mathrm{cap}}\le\sum_{j=1}^N\frac{\ell_j}{s_0}.
\]
\end{enumerate}
\end{theorem}
\begin{proof}
\cref{thm:discrete1}
By Definition~\eqref{def:Wdomain} and \S\ref{sec:Wdomains}, the sets
$\{W(a,b)\}_{(a,b)\in A\times B}$ form a finite partition of $\R^m$ up to a null
set. On each $W(a,b)$ we have $v(x)\equiv b-a$. Since $v$ is constant on each
cell, Carath\'eodory solutions are unique until the first boundary hit.

\cref{thm:discrete2}
On a two-cell interface $\Sigma$, the Filippov set is
$\mathcal F(x)=\co\{v_1,v_2\}$. Under \eqref{eq:conv-normal}, Filippov theory
\cite[Ch.~2]{Filippov1988} yields a unique $v_{\mathrm{slide}}\in\mathcal F(x)$
with $n\cdot v_{\mathrm{slide}}=0$, and local attraction to $\Sigma$.

\cref{thm:discrete3}
Along each $\Sigma_j$, solutions slide with tangential speed bounded below by
$s_0$ by \eqref{eq:tang-speed}. Traversing a segment of length $\ell_j$ therefore
takes at most $\ell_j/s_0$ time. Summation over $j$ gives the bound and finite-time
capture into $W(a^*,a^*)$, where $v\equiv 0$.
\end{proof}

\subsection{Euler-flow approximation and hitting times}
\label{sec:euler-hitting}

The small-step interpretation of RRR as a time-rescaled continuous flow,
observed empirically in \cite[Section~5.5.8]{elser2025solving}, is justified here
by explicit error bounds for the Euler discretization and their consequences for
hitting times.

\noindent For $\eps>0$ and $x_0\in\R^m$, define the RRR iterates
$x_{k+1}=x_k+\eps v(x_k),~ x_0\in\R^m$, and let $x^\eps(\cdot)$ be the piecewise-linear interpolation with
$x^\eps(k\eps)=x_k$.

\begin{theorem}[Euler-flow approximation]\label{thm:euler}
Let $K\subset\R^m$ be compact and $T>0$. Assume either
\begin{enumerate}[(a)]
\item \textnormal{(S1)} holds on an open neighborhood of $K$, or
\item \textnormal{(S2)} holds and the Filippov solution $x(\cdot)$ intersects
only finitely many two-cell interfaces on $[0,T]$ and satisfies
$\dist(x(t),\mathcal J)\ge\delta_0>0$ for all $t\in[0,T]$, where $\mathcal J$
denotes the set of multiway junctions.
\end{enumerate}
Then there exist $\eps_0>0$ and $C=C(K,T)>0$ such that for all
$x_0\in K$ and $0<\eps\le\eps_0$,
\begin{equation}\label{eq:Euler-trajectory}
\sup_{t\in[0,T]}\|x^\eps(t)-x(t)\|\le C\,\eps,
\end{equation}
where $x(\cdot)$ is the Carath\'eodory (resp.\ Filippov) solution of
\eqref{eq:flow} with $x(0)=x_0$.
\end{theorem}

\begin{proof}
Let $e^\eps(t):=x^\eps(t)-x(t)$.  

\emph{Step 1: Lipschitz regions.}
On any open set $\Omega\subset\R^m$ where $v$ is $L$-Lipschitz, $e^\eps$ satisfies
\[
\dot e^\eps(t)=v(x^\eps(t))-v(x(t)),
\qquad
\|\dot e^\eps(t)\|\le L\|e^\eps(t)\|+O(\eps),
\]
for a.e.\ $t$ away from grid points. By Gr\"onwall,
\begin{equation}\label{eq:euler-lip}
\sup_{t\in[0,T]}\|e^\eps(t)\|\le C_1\eps .
\end{equation}
Under \textnormal{(S1)} this holds on a neighborhood of $K$; under
\textnormal{(S2)} it holds on each $W$-domain away from interfaces.

\emph{Step 2: two-cell interfaces.}
Let $\Sigma$ be a $C^1$ interface satisfying \eqref{eq:conv-normal}. By
Theorem~\ref{thm:discrete}\ref{thm:discrete2}, the Filippov solution satisfies
$\dot x(t)=v_{\mathrm{slide}}$ on $\Sigma$. Let $x_k$ be such that the Euler step
crosses $\Sigma$ at some $\tau\in(0,\eps)$. A two-step Taylor expansion of the
Euler scheme yields
\begin{equation}\label{eq:two-step}
x_{k+2}-x_k
=2\eps\,v_{\mathrm{slide}}+O(\eps^2),
\end{equation}
and the normal component of $x^\eps$ relative to $\Sigma$ is $O(\eps^2)$
(cf.\ \cite[Sec.~3]{Brogliato2016}). Hence
\begin{equation}\label{eq:euler-slide}
\sup_{t\in[k\eps,(k+2)\eps]}\|e^\eps(t)\|\le C_2\eps .
\end{equation}

\emph{Step 3: summation.}
By assumption, only finitely many interfaces are crossed on $[0,T]$ and junctions
are avoided. Combining \eqref{eq:euler-lip} and \eqref{eq:euler-slide} over the
resulting partition of $[0,T]$ yields \eqref{eq:Euler-trajectory}.
\end{proof}

\begin{theorem}[Hitting-time convergence]\label{thm:hitting}
Assume \textnormal{(S1)} and let $x_0\in U$, where $U$ is as in
Theorem~\ref{thm:smooth}. For $\delta>0$, define
\[
T^*_\delta(x_0):=\inf\{t\ge0:\|v(x(t))\|\le\delta\},
\qquad
t^*_\delta(\eps;x_0):=\inf\{k\eps:\|v(x_k)\|\le\delta\}.
\]
Then there exist $\eps_0>0$ and $C_\delta>0$ such that for all
$0<\eps\le\eps_0$,
\begin{equation}\label{eq:hitting-error}
\bigl|t^*_\delta(\eps;x_0)-T^*_\delta(x_0)\bigr|\le C_\delta\,\eps .
\end{equation}
\end{theorem}

\begin{proof}
By Theorem~\ref{thm:smooth},
\[
\|v(x(t))\|\le C_0 e^{-\mu t}\|v(x_0)\|,
\qquad
\frac{d}{dt}\|v(x(t))\|\le-\mu\|v(x(t))\|
\]
for a.e.\ $t$ with $v(x(t))\neq0$. Hence
\[
\dot{\|v(x(t))\|}\big|_{t=T^*_\delta(x_0)}\le-\mu\delta<0,
\]
so the crossing of $\{\|v\|=\delta\}$ is transversal. Fix $T>T^*_\delta(x_0)$. By Theorem~\ref{thm:euler},
\[
\sup_{t\in[0,T]}\|x^\eps(t)-x(t)\|\le C\eps .
\]
Continuity of $v$ on $U$ implies the existence of $L_\delta>0$ such that
\[
\sup_{t\in[0,T]}
\bigl|\|v(x^\eps(t))\|-\|v(x(t))\|\bigr|
\le L_\delta C\eps .
\]
Therefore, for $\eps$ sufficiently small, $x^\eps(\cdot)$ crosses the band
$\{\delta/2\le\|v\|\le2\delta\}$ exactly once, and the implicit function theorem
yields the estimate \eqref{eq:hitting-error}.
\end{proof}

\subsection{Quantitative flow limit}

We summarize the asymptotic behavior of the solution time in the flow and its
discrete approximation.

\begin{theorem}[Quantitative flow time and discrete correction]\label{thm:quant}
Assume the hypotheses of Theorem~\ref{thm:smooth}. Then there exist
$c_1,c_2>0$, $\delta_0>0$, and a neighborhood $U$ of $x^*$ such that for all
$x_0\in U$ and $\delta\in(0,\delta_0]$,
\begin{equation}\label{eq:Tdelta-bound}
T^*_\delta(x_0)\ \le\ c_1+c_2\log\!\frac{1}{\delta}.
\end{equation}
Moreover, for the discrete RRR iteration \eqref{eq:rrr},
\begin{equation}\label{eq:t-star-correction}
t^*_\delta(\eps;x_0)
= T^*_\delta(x_0)+O(\eps),
\qquad \eps\downarrow0,
\end{equation}
where the $O(\eps)$ term is uniform for $x_0$ in compact subsets of $U$.
\end{theorem}

\begin{proof}
By Theorem~\ref{thm:smooth}, there exist $C,\mu>0$ and a neighborhood $U$ of $x^*$
such that
\[
\|v(x(t))\|\le C e^{-\mu t}\|v(x_0)\|,
\qquad \forall\,x_0\in U,\ t\ge0.
\]
Fix a compact $K\subset U$ and set
\[
M:=\sup_{x_0\in K}\|v(x_0)\|<\infty.
\]
Evaluating the above bound at $t=T^*_\delta(x_0)$ gives
\[
\delta
\le \|v(x(T^*_\delta(x_0)))\|
\le C e^{-\mu T^*_\delta(x_0)} M,
\]
which implies \eqref{eq:Tdelta-bound} with
$c_2=\mu^{-1}$ and $c_1=\mu^{-1}\log(CM)$. The discrete correction \eqref{eq:t-star-correction} follows directly from the
hitting-time estimate of Theorem~\ref{thm:hitting}, which yields
\[
|t^*_\delta(\eps;x_0)-T^*_\delta(x_0)|\le C_\delta\,\eps
\]
uniformly for $x_0\in K$ and $\eps$ sufficiently small.
\end{proof}

\begin{remark}[Flow-limit behavior of the solution time]\label{rem:flow-limit}
For any fixed $x_0\in U$ and $\delta\in(0,\delta_0]$,
\[
\lim_{\eps\downarrow0} t^*_\delta(\eps;x_0)=T^*_\delta(x_0)<\infty.
\]
Consequently, the flow-time stabilizes as $\eps\downarrow0$, while the iteration
count
\[
k^*_\delta(\eps;x_0):=\frac{t^*_\delta(\eps;x_0)}{\eps}
\]
diverges at rate $O(\eps^{-1})$. This separation explains the empirical
observation that small step sizes regularize solution time while inducing a
problem-dependent iteration-optimal choice of $\eps$.
\end{remark}

\subsection{Discrete Lyapunov inequality for small-step RRR}
\label{sec:disc-lyap}

Let $x^*\in A\cap B$ be a transversal intersection in the smooth regime
\textnormal{(S1)}. Set
\[
\mathcal T_M := T_{x^*}A\cap T_{x^*}B,\qquad
\mathcal T^\perp := (\mathcal T_M)^\perp,
\]
and let $P_{\mathcal T}$ denote the orthogonal projector onto $\mathcal T^\perp$.
Let $J:=Dv(x^*)$ be the Jacobian from Lemma~\ref{lem:jacobian}.

\begin{theorem}[Local discrete Lyapunov inequality]\label{thm:disc-lyap}
Assume \textnormal{(S1)} and let $x^*\in A\cap B$ be transversal. Define
$J_{\mathcal T}:=P_{\mathcal T}JP_{\mathcal T}$. Then there exist
$H_{\mathcal T}\in\R^{m\times m}$, $c>0$, $\eps_0>0$, and $r>0$ such that:
\begin{enumerate}[(i)]
\item\label{disc:HT}
$H_{\mathcal T}\succeq0$, $\ker H_{\mathcal T}=\mathcal T_M$, and
$H_{\mathcal T}|_{\mathcal T^\perp}\succ0$;

\item\label{disc:inv}
for all $0<\eps\le\eps_0$ and $\|x_0-x^*\|\le r$, the RRR iterates
$x_{k+1}=x_k+\eps v(x_k)$ satisfy $\|x_k-x^*\|\le r$ for all $k\ge0$;

\item\label{disc:lyap}
the quadratic form
\begin{equation}\label{eq:V-def}
V(x):=(x-x^*)^\top H_{\mathcal T}(x-x^*)
\end{equation}
satisfies
\begin{equation}\label{eq:disc-lyap}
V(x_{k+1})\le(1-c\eps)\,V(x_k)\qquad\forall k\ge0;
\end{equation}

\item\label{disc:exp}
in particular,
\begin{equation}\label{eq:disc-exp}
V(x_k)\le e^{-c k\eps}\,V(x_0)\qquad\forall k\ge0,
\end{equation}
and $\|P_{\mathcal T}(x_k-x^*)\|\to0$ exponentially in discrete time
$t_k=k\eps$.
\end{enumerate}
\end{theorem}
\begin{proof}
By Lemma~\ref{lem:jacobian} and Theorem~\ref{thm:smooth}, the spectrum of
$J_{\mathcal T}$ satisfies $\Re\lambda(J_{\mathcal T})<0$ on $\mathcal T^\perp$,
and $J_{\mathcal T}=0$ on $\mathcal T_M$. By the continuous-time Lyapunov theorem,
there exists $H_{\mathcal T}$ satisfying \eqref{disc:HT} and
\begin{equation}\label{eq:Lyap-cont}
H_{\mathcal T}J_{\mathcal T}+J_{\mathcal T}^\top H_{\mathcal T}
\;\preceq\;-2\gamma H_{\mathcal T}
\end{equation}
on $\mathcal T^\perp$ for some $\gamma>0$.

Consider the linearized map $y_{k+1}=(I+\eps J)y_k$. Then
\begin{align}
V(y_{k+1})-V(y_k)
&=y_k^\top\!\left[\eps(J^\top H_{\mathcal T}+H_{\mathcal T}J)
+\eps^2J^\top H_{\mathcal T}J\right]\!y_k. \label{eq:Vlin}
\end{align}
On $\mathcal T^\perp$, \eqref{eq:Lyap-cont} and finite dimensionality imply
$J^\top H_{\mathcal T}J\preceq C_1H_{\mathcal T}$ for some $C_1>0$. Choosing
$0<\eps_0\le\gamma/C_1$, we obtain
\[
(I+\eps J)^\top H_{\mathcal T}(I+\eps J)-H_{\mathcal T}
\;\preceq\;-\gamma\eps H_{\mathcal T},
\]
hence $V(y_{k+1})\le(1-\gamma\eps)V(y_k)$ for the linearized system. For the nonlinear RRR iteration, write
\[
x_{k+1}-x^*=(I+\eps J)(x_k-x^*)+\eps r(x_k),
\qquad r(x):=v(x)-J(x-x^*).
\]
Since $v\in C^1$ near $x^*$, there exists $C_2>0$ such that
$\|r(x)\|\le C_2\|x-x^*\|^2$ locally. Expanding $V(x_{k+1})$ and using
\eqref{eq:Vlin} yields
\[
V(x_{k+1})
\le(1-\gamma\eps)V(x_k)+C_3\eps\|x_k-x^*\|^3+C_4\eps^2\|x_k-x^*\|^4 .
\]
Since $H_{\mathcal T}|_{\mathcal T^\perp}\succ0$, there exist $m,M>0$ such that
\[
m\|P_{\mathcal T}(x_k-x^*)\|^2\le V(x_k)\le
M\|P_{\mathcal T}(x_k-x^*)\|^2 .
\]
Choosing $r>0$ sufficiently small ensures
\[
C_3\eps\|x_k-x^*\|^3+C_4\eps^2\|x_k-x^*\|^4
\le\tfrac{\gamma}{2}\eps V(x_k)
\]
whenever $\|x_k-x^*\|\le r$. Thus
\[
V(x_{k+1})\le(1-\tfrac{\gamma}{2}\eps)V(x_k),
\]
which gives \eqref{eq:disc-lyap} with $c=\gamma/2$. Invariance of $B(x^*,r)$
follows by induction, and \eqref{eq:disc-exp} is obtained by iteration.
\end{proof}

\section{Global basins, generic chains, and junctions}\label{sec:beyond-local}

\subsection{Smooth tubular basin}

We strengthen Theorem~\ref{thm:smooth} under uniform geometric control of the
feasible manifold $M:=A\cap B$.

\begin{assumption}[Smooth tube]\label{ass:ST}
Assume \textnormal{(S1)} and that $M:=A\cap B$ is compact. Suppose:
\begin{enumerate}[(i)]
\item $\angle(\mathcal T_xA,\mathcal T_xB)\in[\theta_0,\pi/2]$ for all $x\in M$
and some $\theta_0>0$;
\item $A$ and $B$ have reach $\ge\rho>0$ and second fundamental forms bounded by
$\kappa<\infty$.
\end{enumerate}
\end{assumption}

\noindent Define, for $0<r\le\rho$,
\begin{equation}\label{eq:tube}
\mathcal S_r
:=\{x\in\R^m:\ \dist(x,A)\le r,\ \dist(R_Ax,B)\le r\}.
\end{equation}

\begin{theorem}[Strict Lyapunov function on a tube]\label{thm:tube-Lyap}
Under Assumption~\ref{ass:ST} there exist $r\in(0,\rho]$ and $\sigma>0$ such that
for every solution $x(\cdot)$ of \eqref{eq:flow} with $x(t)\in\mathcal S_r$,
\begin{equation}\label{eq:tube-decay}
\frac{d}{dt}\Bigl(\tfrac12\|v(x(t))\|^2\Bigr)
\le -\sigma\,\|v(x(t))\|^2 .
\end{equation}
Consequently, $\|v(x(t))\|\to0$ exponentially and $\dist(x(t),M)\to0$ as
$t\to\infty$.
\end{theorem}

\begin{proof}
For $\bar x\in M$, Lemma~\ref{lem:jacobian} and Assumption~\ref{ass:ST}(i) yield
\[
\tfrac12\bigl(Dv(\bar x)+Dv(\bar x)^\top\bigr)
\preceq -\sin^2\theta_0\,P_{\mathcal T_{\bar x}A+\mathcal T_{\bar x}B}.
\]
Assumption~\ref{ass:ST}(ii) implies $P_A,P_B\in C^1(\mathcal S_r)$ for $r$ small,
and there exists $C<\infty$ such that
\[
\|Dv(x)-Dv(\bar x)\|\le Cr,
\qquad x\in\mathcal S_r,\ \bar x:=P_M(x).
\]
Choosing $r>0$ so that $Cr\le\tfrac12\sin^2\theta_0$, the symmetric part
$S(x):=\tfrac12(Dv(x)+Dv(x)^\top)$ satisfies
\[
S(x)\preceq -\tfrac12\sin^2\theta_0\,P_{\mathcal T_{\bar x}A+\mathcal T_{\bar x}B},
\qquad \forall x\in\mathcal S_r .
\]
Let $E(x):=\tfrac12\|v(x)\|^2$. Along solutions in $\mathcal S_r$,
\[
\dot E(x(t))
=\langle v(x(t)),Dv(x(t))v(x(t))\rangle
=\langle v(x(t)),S(x(t))v(x(t))\rangle
\le -\tfrac12\sin^2\theta_0\,\|v(x(t))\|^2 ,
\]
which is \eqref{eq:tube-decay} with $\sigma:=\tfrac12\sin^2\theta_0$.
\end{proof}

\begin{corollary}[No recurrence in the tube]\label{cor:no-recurrence}
Consider Assumption~\ref{ass:ST}. Let $x(\cdot)$ be a solution of $\dot x=v(x)$ with
$x(0)\in\mathcal S_r$. Then
\[
\omega(x)\subseteq M:=A\cap B.
\]
In particular, $\mathcal S_r$ contains no nontrivial periodic orbits and no
chaotic invariant sets of the flow.
\end{corollary}

\begin{proof}
By Theorem~\ref{thm:tube-Lyap},
\[
\dot E(x(t))\le -\sigma\|v(x(t))\|^2<0
\quad\text{whenever }v(x(t))\neq0.
\]
Hence $E(x(t))$ is strictly decreasing along any nonconstant trajectory in
$\mathcal S_r$. Therefore no point $x\in\mathcal S_r\setminus M$ can belong to
$\omega(x)$, and $\omega(x)\subseteq\{x:\ v(x)=0\}=M$.
\end{proof}

\begin{remark}[Regularity in the feasible case]\label{rem:chaos}
In the feasible case $A\cap B\neq\varnothing$, Assumption~\ref{ass:ST} and
Theorem~\ref{thm:tube-Lyap} yield a strict Lyapunov function
$E(x)=\tfrac12\|v(x)\|^2$ on $\mathcal S_r$. Consequently, the flow restricted to
$\mathcal S_r$ is gradient-like: all trajectories converge to $M$, and no
recurrent dynamics occur in $\mathcal S_r$.

This is consistent with the observations in \cite[Chapter~5]{elser2025solving},
where chaotic behavior for RRR is reported only for infeasible instances
($A\cap B=\varnothing$) or outside basins of attraction. The present results are
local: they exclude recurrence near $M$ but do not address global dynamics.
\end{remark}

\subsection{Generic descent chains in the discrete case}

\begin{assumption}[Discrete genericity]\label{ass:DG}
Assume \textnormal{(S2)} and:
\begin{enumerate}[(i)]
\item $\{\|b-a\|^2:\ (a,b)\in A\times B\}$ are pairwise distinct;
\item for each $b\in B$ (resp.\ $a\in A$), the Voronoi adjacency graph on $A$
(resp.\ $B$) is connected;
\item W-domain boundaries are in general position; in particular, all multiway
junctions have codimension $\ge2$.
\end{enumerate}
\end{assumption}

For $(a,b)\in A\times B$, define the \emph{pair distance}
\[
d(a,b):=\|b-a\|^2 .
\]

\begin{lemma}[Descent across an A-switch]\label{lem:A-descent}
Let $b\in B$ and let $\Sigma\subset\partial W(a_1,b)\cap\partial W(a_2,b)$ be an
A-switch boundary. Let $n$ be the unit normal to $\Sigma$ oriented from
$W(a_1,b)$ to $W(a_2,b)$, and set $v_i:=b-a_i$. Define
\[
\alpha:=\langle a_2-a_1,\ b-a_1\rangle .
\]
Then:
\begin{enumerate}[(i)]
\item $(n\!\cdot\! v_1)(n\!\cdot\! v_2)<0
\iff
\alpha\in(0,\|a_2-a_1\|^2)$;
\item along the sliding velocity $v_{\mathrm{slide}}$ on $\Sigma$,
\[
\frac{d}{dt}d(a(t),b)<0
\quad\text{whenever }\alpha>\tfrac12\|a_2-a_1\|^2,
\]
where $a(t)\in\{a_1,a_2\}$ denotes the active $A$-index along the sliding path.
\end{enumerate}
\end{lemma}

\begin{proof}
Since $\Sigma$ is a Voronoi bisector between $a_1$ and $a_2$ with $b$ fixed, its
normal satisfies $n=\pm(a_2-a_1)/\|a_2-a_1\|$. Direct evaluation of
$n\cdot v_i$ yields (i). Statement (ii) follows since $d(a,b)$ depends only on the
active index $a$, and the inequality $\alpha>\tfrac12\|a_2-a_1\|^2$ selects the
direction corresponding to the strictly smaller distance.
\end{proof}

\begin{theorem}[Finite descent chain to a solution cell]\label{thm:chain}
Assume Assumption~\ref{ass:DG} and $A\cap B\neq\varnothing$. Then for any
nonsolution cell $W(a,b)$ there exists a finite sequence
\[
W(a,b)=W(a_0,b_0)\to W(a_1,b_1)\to\cdots\to W(a_N,b_N)=W(a^*,a^*)
\]
such that:
\begin{enumerate}[(i)]
\item $W(a_j,b_j)$ and $W(a_{j+1},b_{j+1})$ share a boundary satisfying
\eqref{eq:conv-normal};
\item $d(a_{j+1},b_{j+1})<d(a_j,b_j)$ for all $j$, and
\[
N\le\frac{d(a,b)-d(a^*,a^*)}{\Delta},
\qquad
\Delta:=\min_{\substack{(a',b')\neq(a'',b'')\\(a',b'),(a'',b'')\in A\times B}}
|d(a',b')-d(a'',b'')|>0 .
\]
\end{enumerate}
Consequently, any Filippov solution following these convergent boundaries reaches
$W(a^*,a^*)$ after finitely many boundary transitions.
\end{theorem}

\begin{proof}
If $a\notin\arg\min_{a'\in A}d(a',b)$, Assumption~\ref{ass:DG}(ii) provides a
Voronoi path in $A$ to some $a'$ with $d(a',b)<d(a,b)$; by
Lemma~\ref{lem:A-descent}, one can choose A-switches along which $d$ strictly
decreases. If $b\notin\arg\min_{b'\in B}d(a,b')$, the symmetric argument applies.
Since all values of $d$ are distinct and bounded below by $d(a^*,a^*)$, the
number of strict decreases is finite and bounded by $(d(a,b)-d(a^*,a^*))/\Delta$.
\end{proof}

\subsection{Nonunique projections and junctions}

Let $\mathcal M_A,\mathcal M_B\subset\R^m$ denote the medial axes where $P_A,P_B$
are multivalued, and let $\mathcal J$ denote the set of points where three or more
$W$-domains intersect.

\begin{theorem}[Negligibility of medial and junction sets]\label{thm:medial}
Assume either \textnormal{(S1)} with compact $M=A\cap B$ or \textnormal{(S2)} with
Assumption~\ref{ass:DG}. Then:
\begin{enumerate}[(i)]
\item
\[
\mathcal L^m(\mathcal M_A)=\mathcal L^m(\RA^{-1}(\mathcal M_B))=0 ;
\]

\item
in case \textnormal{(S2)},
\[
\mathcal L^m(\mathcal J)=0,
\qquad
\mathcal L^1\!\left(\{t\ge0:\ x(t)\in\mathcal J\}\right)=0
\]
for every Filippov solution $x(\cdot)$.
\end{enumerate}
\end{theorem}

\begin{proof}
(i)
For closed $C^2$ submanifolds (and for finite sets), medial axes have Hausdorff
codimension at least $1$, hence zero Lebesgue measure. Since $\RA$ is a $C^1$
diffeomorphism on $\mathcal S_r$ in case \textnormal{(S1)} and an affine map in
case \textnormal{(S2)} (Definition~\eqref{def:Wdomain}), the preimage
$\RA^{-1}(\mathcal M_B)$ also has measure zero.

(ii)
Under Assumption~\ref{ass:DG}(iii), $\mathcal J$ is a finite union of intersections
of at least two smooth codimension-$1$ hypersurfaces and therefore has
codimension at least $2$, implying $\mathcal L^m(\mathcal J)=0$. Since Filippov
solutions are absolutely continuous, their preimages of $\mathcal J$ have
one-dimensional Lebesgue measure zero.
\end{proof}

\subsection{A percolation and renormalization group model for the $\beta\uparrow1$ catastrophe}
\label{sec:beta_catastrophe_model}

This subsection is heuristic. Its role is to formalize a mesoscopic object
that can carry the $\beta\uparrow1$ divergence picture discussed in
\cite[Ch.~5]{elser2025solving} and measured in \S\ref{sec:num_ledm}.

\paragraph{Cell dynamics induced by $T_\beta$.}
Assume \textnormal{(S2)} and use the $W$-domain partition from
\eqref{def:Wdomain}. Index the (a.e.) cells by a finite set $V$,
\[
\R^m=\bigsqcup_{i\in V} W_i \quad (\text{up to a null set}),\qquad
v(x)=v_i \ \ \text{for a.e.\ }x\in W_i,
\]
so that the RRR map is
\[
T_\beta(x)=x+\beta v_i,\qquad x\in W_i.
\]
Fix a probability measure $\mu$ on $\R^m$ with continuous strictly positive
density. The quantity below measures how much of cell $i$ is transported into
cell $j$ by one step of $T_\beta$.

\begin{definition}[Cell transition kernel, support digraph, and absorbing set]
\label{def:cell-kernel}
For $\beta>0$, define
\[
P_\beta(i,j)
:=\mu\!\left(W_i\cap T_\beta^{-1}(W_j)\,\middle|\,W_i\right)
=\frac{\mu\!\left(W_i\cap (W_j-\beta v_i)\right)}{\mu(W_i)}.
\]
Define the support digraph $G_\beta=(V,E_\beta)$ by
\[
(i\to j)\in E_\beta \iff P_\beta(i,j)>0 .
\]
Let
\[
S:=\{\,i\in V:\ v_i=0\,\}
\]
denote the set of solution cells (absorbing fixed points of $T_\beta$).
\end{definition}

The deterministic orbit $x_{k+1}=T_\beta(x_k)$ induces a cell itinerary
$i_k\in V$ via $x_k\in W_{i_k}$.
The kernel $P_\beta$ is a \emph{coarse} model of cell-to-cell transport under a
reference distribution; $G_\beta$ records which transitions are possible at all.
Large strongly connected regions of $G_\beta$ represent a mechanism for long
wandering before hitting $S$, consistent with the “search phase” dominating
in \S\ref{sec:num_ledm}.

\paragraph{Order parameter.}
Consider a problem family indexed by $n$ with cell sets $V_n$ and solution sets
$S_n\subseteq V_n$. Write $G_{\beta,n}$ for the corresponding support digraph.

\begin{definition}[Wandering order parameter and graph phase transition]
\label{def:order-parameter}
Let $\mathrm{SCC}_{\max}(G)$ be a largest strongly connected component of $G$.
Define
\[
\Phi_n(\beta)
:=\frac{1}{|V_n|}
\Bigl|\mathrm{SCC}_{\max}\bigl(G_{\beta,n}\restriction(V_n\setminus S_n)\bigr)\Bigr|.
\]
We say the family exhibits a (directed) phase transition at $\beta_c\in(0,1)$ if
\[
\lim_{n\to\infty}\Phi_n(\beta)=0 \ \ (\beta<\beta_c),
\qquad
\liminf_{n\to\infty}\Phi_n(\beta)>0 \ \ (\beta>\beta_c).
\]
\end{definition}

\noindent To connect $\beta\mapsto G_{\beta,n}$ with standard threshold heuristics, we
randomize edges according to $P_\beta$ as an inhomogeneous directed percolation
model \cite{DurrettRandomGraphDynamics,GrimmettPercolation}.

\begin{definition}[Edge percolation on the support digraph]
\label{def:edge-perc}
Let $\{U_{ij}\}_{(i,j)\in V\times V}$ be i.i.d.\ $\mathrm{Unif}[0,1]$.
Define $G_\beta^\omega$ by declaring $i\to j$ open iff $U_{ij}\le P_\beta(i,j)$.
Equivalently, conditional on $P_\beta$, edges are independent with
$\mathbb{P}(i\to j\text{ open})=P_\beta(i,j)$.
\end{definition}

\begin{lemma}[Monotone coupling in $\beta$]
\label{lem:monotone-coupling}
If $P_\beta(i,j)$ is nondecreasing in $\beta$ for all $(i,j)$, then there exists
a coupling such that
\[
\beta_1\le\beta_2 \implies E(G_{\beta_1}^\omega)\subseteq E(G_{\beta_2}^\omega)
\quad\text{a.s.}
\]
\end{lemma}

\begin{proof}
Use common uniforms $U_{ij}$ and the inclusion
$\{U_{ij}\le P_{\beta_1}(i,j)\}\subseteq \{U_{ij}\le P_{\beta_2}(i,j)\}$.
\end{proof}

\paragraph{Renormalization group coarse-graining.}
We coarse-grain $G$ by collapsing its recurrent blocks (SCCs), in the spirit of
RG block transformations (cf.\ \cite{goldenfeld1992lectures}).

\begin{definition}[SCC condensation RG map]
\label{def:rg}
Let $\mathcal{C}(G)$ be the condensation digraph of $G$ (vertices are SCCs; edges
record inter-SCC reachability). Define
\[
\mathcal{R}(G):=\mathcal{C}(G),\qquad G^{(0)}=G,\ \ G^{(\ell+1)}=\mathcal{R}(G^{(\ell)}).
\]
\end{definition}

\begin{lemma}[Condensation preserves reachability]
\label{lem:condensation-reach}
For any digraph $G$ and vertices $u,v$,
\[
u\leadsto v \text{ in }G \iff \mathrm{SCC}(u)\leadsto \mathrm{SCC}(v)\text{ in }\mathcal{C}(G).
\]
\end{lemma}

\begin{proof}
Map a path to SCC labels and delete repeats; conversely concatenate inter-SCC
edges with intra-SCC connectivity.
\end{proof}

\begin{conjecture}[$\beta$-catastrophe as emergence of a giant wandering SCC]
\label{conj:beta-catastrophe}
For certain hard discrete families, there exists $\beta_c\in(0,1)$ such that the
order parameter $\Phi_n(\beta)$ from Definition~\ref{def:order-parameter}
undergoes a phase transition at $\beta_c$. For $\beta>\beta_c$, typical
trajectories exhibit long transients supported by the giant SCC of
$G_{\beta,n}\restriction(V_n\setminus S_n)$ before reaching $S_n$, causing rapid
growth (or divergence) of solution times as $\beta\uparrow1$.
\end{conjecture}

The heatmap diagnostics done in \S\ref{sec:num_ledm} can be viewed as empirical
proxies for (i) accessibility of $S_n$ (entry probability) and (ii) recurrence
within $V_n\setminus S_n$ (mesoscopic recurrence index), both of which are
predicted to change sharply across $\beta_c$ under
Conjecture~\ref{conj:beta-catastrophe}.

\section{Numerical Evidences}
We consider the foundational framework in \cite{elser2021learning}, where the learning problem was framed as feasibility formulation. The training subproblem for a single sample is the feasibility problem
\[
\text{find } z\in \bigcap_{v\in V} C_v.
\] For a mini-batch $\{(x^{(i)},y^{(i)})\}_{i=1}^B$, the framework constructs a single enlarged computation graph by
replicating all \emph{data-dependent} nodes per sample while \emph{sharing} parameter nodes across samples.
Thus, parameter nodes have outgoing edges into every sample replica, and $C_p^{\mathrm{par}}$ enforces cross-sample
consensus automatically. The resulting feasibility problem can be written as
\[
\text{find } z\in
\left(\bigcap_{i=1}^B\ \bigcap_{v\in V_{\mathrm{data}}^{(i)}} C_v^{(i)}\right)
\cap
\left(\bigcap_{p\in V_{\mathrm{par}}} C_p^{\mathrm{par}}\right).
\]
There is a single edge-state $z$ on a single
(batch-expanded) graph; the only batch coupling is through the shared parameter-node constraints. We provide an equivalence interpretation that is useful to the set-valued and variational analysis, operator splitting, and machine learning communities.
\begin{theorem}[Two-level wire feasibility and canonical finite-sum inclusion]
\label{thm:twolevel-wire-inclusion}

Fix $N\in\mathbb{N}$.
Let $\mathcal{V}_{\mathrm{data}}$ be a finite index set of local relation types,
$\mathcal{W}_{\mathrm{act}}$ a finite index set of activation-wire types,
and $\Theta$ a finite index set of shared parameter wires.
For $w\in\mathcal{W}_{\mathrm{act}}$ let $H_w\simeq\R^{d_w}$, and for $\theta\in\Theta$ let
$H_\theta\simeq\R^{d_\theta}$.

\begin{enumerate}[label=\textup{(\roman*)},leftmargin=2.2em]

\item \textbf{(Local-copy wire lift).}
For each replica $i\in[N]$ and each relation $v\in\mathcal{V}_{\mathrm{data}}$, let
$\mathcal{W}(v)\subseteq \mathcal{W}_{\mathrm{act}}\cup\Theta$ and define
\[
H_{i,v} := \prod_{\omega\in\mathcal{W}(v)} H_\omega,
\qquad
M_{i,v}\subseteq H_{i,v}.
\]
Set
\[
\mathcal{H}
:= \prod_{i=1}^N \ \prod_{v\in\mathcal{V}_{\mathrm{data}}} H_{i,v},
\qquad
x=(x_{i,v})\in\mathcal{H}.
\]

\item \textbf{(Set-feasibility formulation).}
Define the product of local relations
\[
\mathcal{M} := \prod_{i=1}^N \ \prod_{v\in\mathcal{V}_{\mathrm{data}}} M_{i,v} \subseteq \mathcal{H}.
\]
For each $(i,w)\in[N]\times\mathcal{W}_{\mathrm{act}}$, let $\Delta_{i,w}$ be the diagonal subspace enforcing
equality of all copies of wire $(i,w)$ across $\{x_{i,v}\}_v$.
For each $\theta\in\Theta$, let $\Delta_\theta$ enforce equality of all copies of $\theta$ across all replicas.
Define
\[
\mathcal{D}_{\mathrm{wire}}
:= \Big(\prod_{i=1}^N \ \prod_{w\in\mathcal{W}_{\mathrm{act}}} \Delta_{i,w}\Big)
   \times
   \Big(\prod_{\theta\in\Theta} \Delta_\theta\Big)
\subseteq \mathcal{H}.
\]
The two-level wire feasibility problem is
\[
\text{find } x\in \mathcal{M}\cap \mathcal{D}_{\mathrm{wire}}.
\]

\item \textbf{(Finite-sum operator inclusion).}
Assume all $M_{i,v}$ are nonempty and closed, and let $N_M$ denote the normal cone of $M$.
Then the feasibility problem in (ii) is equivalent to the inclusion
\begin{equation}
\label{eq:twolevel-wire-inclusion}
0 \in 
\sum_{i=1}^N \ \sum_{v\in\mathcal{V}_{\mathrm{data}}} N_{M_{i,v}}(x)
\;+\;
\sum_{i=1}^N \ \sum_{w\in\mathcal{W}_{\mathrm{act}}} N_{\Delta_{i,w}}(x)
\;+\;
\sum_{\theta\in\Theta} N_{\Delta_{\theta}}(x),
\qquad x\in\mathcal{H},
\end{equation}
where each normal cone acts on its own coordinate block in $\mathcal{H}$ (and as zero elsewhere).
In particular, $J_{N_{M_{i,v}}}=P_{M_{i,v}}$ and $J_{N_{\Delta}}=P_{\Delta}$, where
$P_{\Delta_{i,w}}$ is within-replica averaging and $P_{\Delta_\theta}$ is cross-replica averaging.

\end{enumerate}
\end{theorem}

\begin{corollary}[One-layer special case]
\label{cor:one-layer}

In the one-layer case, the index set of local relations can be taken as a Cartesian grid
$\mathcal{V}_{\mathrm{data}}=[n]$ and replicas as $[N]=[m]$, so that the local constraints are naturally indexed
by pairs $(i,v)\in[m]\times[n]$.
Then Theorem~\ref{thm:twolevel-wire-inclusion}, yields a feasibility model consisting of $mn$ local relation blocks
together with diagonal (consensus) constraints enforcing equality of replicated variables \cite{lal2023nonconvex}. 
\end{corollary}

\subsection{Numerical evidence on one-layer LEDM instance}
\label{sec:num_ledm}

\begin{figure}[t]
\centering
\begin{subfigure}[t]{0.48\textwidth}
  \centering
  \includegraphics[width=\textwidth]{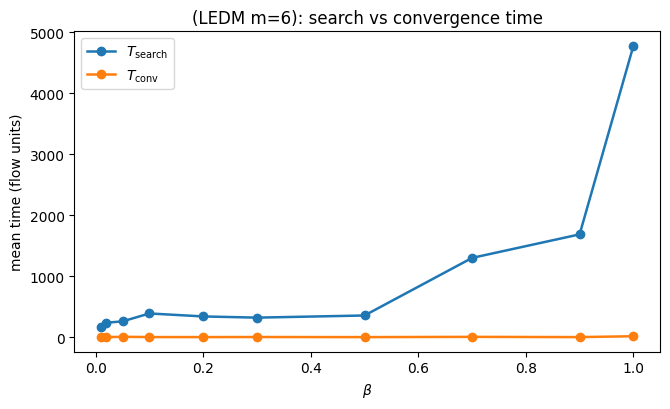}
  \caption{Search and convergence times $T_{\mathrm{search}}(\beta)$ and
  $T_{\mathrm{conv}}(\beta)$.}
\end{subfigure}\hfill
\begin{subfigure}[t]{0.48\textwidth}
  \centering
  \includegraphics[width=\textwidth]{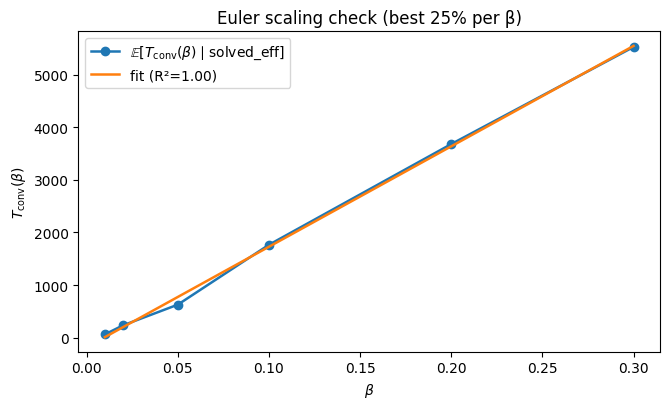}
  \caption{Euler scaling of
  $\mathbb{E}[T_{\mathrm{conv}}(\beta)\mid\mathrm{solved}]$ for $\beta\le 0.3$.}
\end{subfigure}
\caption{Local versus mesoscopic time scales in LEDM.}
\label{fig:exp2-exp3}
\end{figure}
LEDM (Low-rank Euclidean Distance Matrix) is a one-layer nonnegative matrix
factorization feasibility problem in which a symmetric matrix
$Y\in\mathbb{R}^{m\times m}$ with entries
$Y_{ij}=\bigl((i-j)/(m-1)\bigr)^2$ is factorized as $Y\approx WZ^\top$ with
$W,Z\ge 0$ and fixed rank. This fits the one-layer special case of
Corollary~\ref{cor:one-layer}, and hence the two-level wire feasibility
formulation of Theorem~\ref{thm:twolevel-wire-inclusion}.
We run the RRR iteration with hyperparameters fixed as in \cite{elser2021learning}:
batch size $B=m$, rank $r_+=m-1$, normalization $\Omega\in[0.6,0.9]$ depending on
$m$, and iteration budgets up to $k_{\max}=2\times 10^4$ to $1.2\times 10^5$.
The observable is the reconstruction residual
$\mathrm{err}(x_k)$. All quantities are averaged
over multiple random initializations. The experiments are intended to
characterize dynamical regimes rather than to benchmark performance.

\paragraph{Two-phase decomposition.}
Fix tolerances $0<\delta_{\mathrm{solve}}<\delta_{\mathrm{enter}}$ with
$\delta_{\mathrm{enter}}=10^{-2}$ and $\delta_{\mathrm{solve}}=3\times 10^{-2}$.
Define the discrete hitting indices
\[
k_{\mathrm{enter}}(\beta):=\inf\menge{k\ge 0}{\mathrm{err}(x_k)\le\delta_{\mathrm{enter}}},
\qquad
k_{\mathrm{solve}}(\beta):=\inf\menge{k\ge 0}{ \mathrm{err}(x_k)\le\delta_{\mathrm{solve}}},
\]
and the corresponding flow-time quantities
\[
T_{\mathrm{search}}(\beta):=\beta\,k_{\mathrm{enter}}(\beta),
\qquad
T_{\mathrm{conv}}(\beta):=\beta\bigl(k_{\mathrm{solve}}(\beta)-k_{\mathrm{enter}}(\beta)\bigr)_+ .
\]
Figure~\ref{fig:exp2-exp3} (left) reports empirical means of
$T_{\mathrm{search}}(\beta)$ and $T_{\mathrm{conv}}(\beta)$.
The data show that $T_{\mathrm{search}}$ dominates the total runtime across all
$\beta$, while $T_{\mathrm{conv}}$ remains comparatively small whenever entry
occurs. This directly reflects the locality of the flow analysis in
\S\ref{sec:beyond-local}: Theorems~\ref{thm:smooth}-\ref{thm:euler} govern only
the post-entry regime.

\paragraph{Euler scaling in the regular regime.}
Conditioning on successful entry/solve, we restrict to a small-$\beta$ window
$\beta\le 0.3$ and fit the conditional mean convergence time
\[
\mathbb{E}\!\left[T_{\mathrm{conv}}(\beta)\mid \text{solved}\right]
\;\approx\; a+b\,\beta .
\]
Figure~\ref{fig:exp2-exp3} (right) displays the fit, with coefficient of
determination $R^2\approx 1$. This linear scaling is consistent with the
first-order Euler discretization of the limiting flow
(Theorem~\ref{thm:euler}) and the local contraction properties established in
Theorem~\ref{thm:smooth} and Theorem~\ref{thm:disc-lyap}.

\paragraph{Non-local diagnostics over problem size.}
To characterize dynamics prior to entry, for each $(m,\beta)$ define the entry
probability
\[
p_{\mathrm{enter}}(m,\beta)
:=\mathbb{P}\!\left(k_{\mathrm{enter}}(\beta)<k_{\max}\right),
\]
estimated empirically under a fixed iteration budget $k_{\max}$.
Figure~\ref{fig:exp6-heatmaps} (left) displays $p_{\mathrm{enter}}(m,\beta)$.

To quantify non-local wandering before entry, log the error trace
$\{\mathrm{err}(x_k)\}_{k=1}^T$ and define a coarse state
\[
s_k:=\pi(\mathrm{err}(x_k))\in\{0,\dots,B-1\},\qquad
\pi(r):=\mathrm{bin}\bigl(\log_{10} r\bigr),
\]
using quantile binning. The mesoscopic recurrence index is
\[
R(m,\beta)
:=\mathbb{E}\!\left[\frac{1}{T-w}\sum_{k=w+1}^{T}
\mathbf{1}\{s_k\in\{s_1,\dots,s_{k-1}\}\}\right],
\]
with fixed burn-in $w$. Figure~\ref{fig:exp6-heatmaps} (right) displays
$R(m,\beta)$. Larger values indicate increased recurrent or mixing behavior
prior to entry, complementing the local flow description.

\begin{figure}[t]
\centering
\begin{subfigure}[t]{0.5\textwidth}
  \centering
  \includegraphics[width=\textwidth]{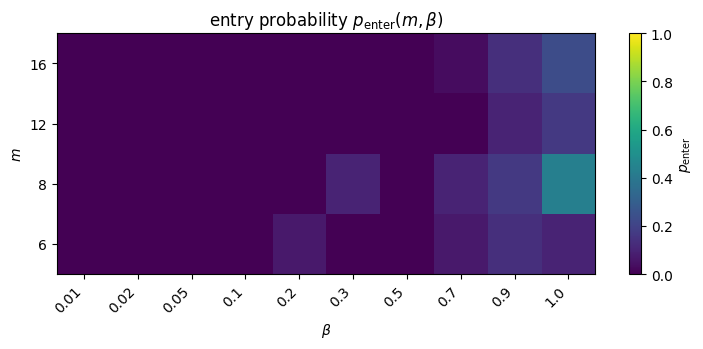}
  \caption{Entry probability $p_{\mathrm{enter}}(m,\beta)$.}
\end{subfigure}\hfill
\begin{subfigure}[t]{0.5\textwidth}
  \centering
  \includegraphics[width=\textwidth]{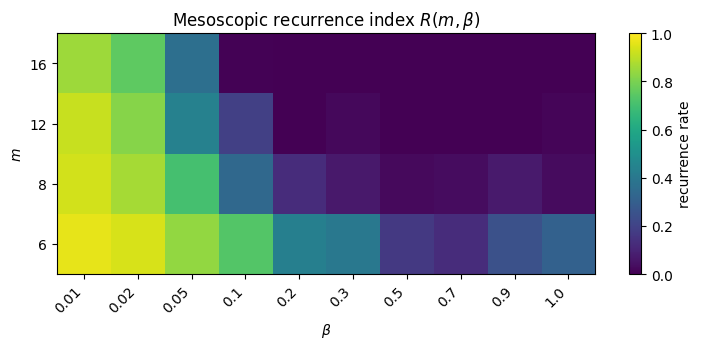}
  \caption{Mesoscopic recurrence index $R(m,\beta)$.}
\end{subfigure}
\caption{Non-local diagnostics across problem size and step size.}
\label{fig:exp6-heatmaps}
\end{figure}

Thus, figures~\ref{fig:exp2-exp3} and \ref{fig:exp6-heatmaps} jointly suggests that:
local flow-limit theory accurately predicts $T_{\mathrm{conv}}$ once entry
occurs, while global performance and the $\beta\to 1$ deterioration are governed
by non-local, instance-dependent entry dynamics, as anticipated in
\S\ref{sec:beta_catastrophe_model}.

\section{Conclusion and outlook}\label{sec:conclusion}

This paper develops a rigorous dynamical-systems description of the
Reflect-Reflect-Relax (RRR) algorithm in the small-step regime, centered on
the limiting flow $\dot x=v(x)$ governing the behavior of RRR as
$\beta\downarrow0$.

In the smooth transversal setting \textnormal{(S1)}, each feasible intersection
$x^*\in A\cap B$ is shown to be a hyperbolic sink of the flow, with strictly
negative real spectrum of $Dv(x^*)$ on the transverse subspace and exponential
decay of the gap $\gap(x)=\|v(x)\|$ (Theorem~\ref{thm:smooth}). Under uniform
geometric assumptions, this local structure extends to a positively invariant
tubular neighborhood $\mathcal T_r$ of the feasible manifold, where
$E(x)=\tfrac12\|v(x)\|^2$ is a strict Lyapunov function and no recurrent or
chaotic dynamics occur (Theorem~\ref{thm:tube-Lyap},
Corollary~\ref{cor:no-recurrence}).

In the discrete regime \textnormal{(S2)}, the limiting vector field is
piecewise constant on $W$-domains. We showed that convergent two-cell
boundaries support Filippov sliding with explicit sliding velocity, and that,
under genericity assumptions, finite descent chains of $W$-domains lead to
finite-time capture into a solution cell (Theorem~\ref{thm:discrete}).

The continuous and discrete descriptions are connected by proving that
small-step RRR is a forward-Euler discretization of the flow with
$O(\beta)$ trajectory error on finite time intervals, including across sliding
segments (Theorem~\ref{thm:euler}). As a consequence, hitting times of any fixed
gap threshold $\delta>0$ satisfy
\[
t^*_\delta(\beta;x_0)\;\to\;T^*_\delta(x_0)\quad\text{as }\beta\downarrow0,
\]
while the iteration count diverges like $1/\beta$
(Theorems~\ref{thm:hitting} and~\ref{thm:quant}). A discrete Lyapunov inequality
further yields exponential convergence rates for the iterates consistent with
the flow analysis (Theorem~\ref{thm:disc-lyap}).

Numerical experiments on a one-layer LEDM instance support this picture
(\S\ref{sec:num_ledm}), showing that local convergence follows the predicted
Euler scaling once a regular regime is reached, while overall runtime is
dominated by a distinct mesoscopic entry phase whose dependence on $\beta$ is
non-local and instance-dependent. This separation clarifies both the scope of
the flow-limit theory and the origin of the deterioration observed near the
Douglas-Rachford limit $\beta\uparrow1$.

\paragraph{Outlook.} Extending the analysis to global phase portraits of the RRR flow and map,
particularly in the infeasible case $A\cap B=\varnothing$ where numerical
evidence suggests complex invariant sets, chaotic attractors, and long
transients, remains open. In the discrete setting, while generic finite descent
chains were established, a quantitative theory relating geometric properties of
the $W$-domain adjacency graph (connectivity, diameter, expansion, recurrent
classes) to entry times, mixing times, coupling times, or hitting-time
distributions is largely undeveloped. Moreover, numerical evidence indicates
that the regime $\beta\uparrow1$ is governed by non-local,
instance-dependent wandering and metastability not captured by local flow
analysis; developing rigorous criteria for such $\beta$-driven transitions and
relating them to coarse-grained or renormalized descriptions of the induced cell
dynamics is an important direction for future work.

\bibliographystyle{plain}
\bibliography{ref}

\end{document}